\documentclass[12pt,a4paper,leqno]{amsart}
\usepackage[utf8]{inputenc}
\pdfoutput=1

\usepackage{amssymb,amsfonts,mathtools,mathrsfs}   
\usepackage{graphicx}
\usepackage{placeins}
\usepackage{enumerate}
\usepackage{tikz}
\usepackage{setspace}
\usepackage{xcolor}
\usepackage{hyperref}
\usepackage{theoremref}

\usepackage[
  margin=1.5cm,
  includefoot,
  footskip=30pt,
]{geometry}

\usepackage[
backend=bibtex,
style=alphabetic,
sortcites=true,
giveninits=true,
maxbibnames=99, 
]{biblatex}
\newtheorem{thm}{Theorem}[section]
\newtheorem{cor}[thm]{Corollary} 
\newtheorem{lem}[thm]{Lemma} 
\newtheorem{prop}[thm]{Proposition}

\theoremstyle{definition}

\theoremstyle{remark}

\numberwithin{equation}{section}

\newcommand{\R}{\mathbb{R}}
 
\newcommand{\Z}{\mathbb{Z}}

\newcommand{\Sph}{\mathbb{S}}

\renewcommand{\div}{\operatorname{div}}

\newcommand{\dd}{\mathop{}\!d} 
 
\newcommand{\mres}{%
  \,\raisebox{-.127ex}{\reflectbox{\rotatebox[origin=br]{-90}{$\lnot$}}}\,%
} 

\def\XXint#1#2#3{{\setbox0=\hbox{$#1{#2#3}{\int}$}
\vcenter{\hbox{$#2#3$}}\kern-.5\wd0}}

\newenvironment{PDE}
	{ \left \{
	\begin{array}{r@{ \ }l @{\quad \: \;} l}
	}
	{
	\end{array} \right . 
	}

\begin{document}

\title[Density estimates for sets of zero \(s\)-mean curvature]{Density estimates and the fractional Sobolev inequality for sets of zero \(s\)-mean curvature}

\author[J. Thompson]{Jack Thompson}
\address{The University of Western Australia (M019)35 Stirling Highway PERTH WA 6009 Australia}
\email{jack.thompson@research.uwa.edu.au}


\date{\today}


\begin{abstract}
We prove that measurable sets \(E\subset \R^n\) with locally finite perimeter and zero \(s\)-mean curvature satisfy the surface density estimates: \begin{align*}
    \operatorname{Per} (E; B_R(x)) \geq CR^{n-1}
\end{align*} for all \(R>0\), \(x\in \partial^\ast E\). The constant \(C\) depends only on \(n\) and \(s\), and remains bounded as \(s\to 1^-\). As an application, we prove that the fractional Sobolev inequality holds on the boundary of sets with zero \(s\)-mean curvature. 
\end{abstract}

\maketitle

\section{Introduction and main results}

A fundamental result in the regularity of minimal surfaces is that a measurable set \(E\subset \R^n \) that is locally perimeter minimising satisfies the thick density estimates: \begin{align}
         \vert E \cap B_R(x) \vert \geq C R^n \text{ and }  \vert B_R(x)  \setminus E \vert \geq C R^n \label{gJJA6X52}
\end{align} for all \(R>0\), \(x\in \partial^\ast E\) where \(B_R(x)\) is a ball in \(\R^n\) centred at \(x\) with radius \(R\) and \(\vert F\vert\) denotes the Lebesgue measure of \(F\) in \(\R^n\), see for example \cite[Section 16.2]{maggi_sets_2012}. Here the constant \(C>0\) depends only on \(n\). Since any sufficiently regular set \(E\) satisfies~\eqref{gJJA6X52} for \(R\) sufficiently small and with the constant depending on the \(C^2\) regularity of \(E\) at \(x\), the non-triviality of~\eqref{gJJA6X52} is that it holds for all \(R>0\) and that the constant is universal. The assumption that \(E\) is locally minimising is essential to the proof of~\eqref{gJJA6X52} since the proof takes advantage of the fact that \(E\cap \partial B_R(x)\) is a competitor for \(\partial^\ast E \cap B_R(x)\).

If one only assumes that \(E\) has zero mean curvature, that is, it is locally stationary with respect to the perimeter functional then one can obtain the surface density estimate: \begin{align}
    \operatorname{Per}(E; B_R(x)) \geq C R^{n-1} \label{jR0GCsky}
\end{align} for all \(R>0\), \(x\in \partial^\ast E\), and with \(C\) depending only on \(n\). Indeed, the monotonicity formula for sets of zero mean curvature immediately implies that \begin{align*}
      \frac{\operatorname{Per}(E; B_R(x))}{R^{n-1}} \geq \lim_{r\to 0^+ }\frac{\operatorname{Per}(E; B_r(x))}{r^{n-1}} =  n\omega_n 
\end{align*} where \(\omega_n = \vert B_1 \vert\), see \cite[Theorem 17.16]{maggi_sets_2012}. Observe that the thick density estimates~\eqref{gJJA6X52} directly imply the surface density estimates~\eqref{jR0GCsky} via the restricted isoperimetric inequality \begin{align*}
    \operatorname{Per}(E;B_R(x) ) \geq C \big ( \min \{ \vert E \cap B_R(x) \vert,\vert B_R(x)  \setminus E \} \big )^{\frac{n-1}n }, 
\end{align*} but sets with zero mean curvature do not necessarily satisfy~\eqref{gJJA6X52}, for example, if \(E = \R^{n-1} \times (0,1)\) then \(\partial E\) has zero mean curvature, but \(\vert E \cap B_R \vert \leq \vert B^{n-1}_R \times (0,1) \vert \leq C R^{n-1} \). An example of a connected hypersurface with zero mean curvature that does not satisfy~\eqref{gJJA6X52} is given by the catenoid.

In the celebrated paper \cite{caffarelli_nonlocal_2010}, a nonlocal (or fractional) concept of perimeter, known as the \(s\)-perimeter, was introduced that was a natural generalisation of the classical perimeter. As well as motivating and defining the \(s\)-perimeter, Caffarelli, Roquejoffre, and Savin in \cite{caffarelli_nonlocal_2010} establish several foundational results for the regularity theory of minimisers\footnote{\label{note1}For a precise definition, see Section 2} of the \(s\)-perimeter including the existence of minimisers, improvement of flatness, a monotonicity formula, blow-up limits to \(s\)-minimal cones, and Federer's dimension reduction. In particular, they also established minimisers of the \(s\)-perimeter satisfy the thick density estimates~\eqref{gJJA6X52}. Since their paper, the regularity theory of critical points of the \(s\)-perimeter (stationary, stable, and minimisers) has become an incredibly active area of research. Some important results include: interior higher regularity \cite{barrios_bootstrap_2014}; boundary regularity and stickiness \cite{dipierro_boundary_2017,dipierro_nonlocal_2020}; classification of cones for \(n=2\) \cite{savin_regularity_2013}; classification of cones for \(s\) close to \(1\) \cite{caffarelli_regularity_2013}; regularity of \(s\)-minimal graphs \cite{cabre_gradient_2019}; explicit examples of sets of zero nonlocal mean curvature \cite{cozzi_growth_2020,davila_nonlocal_2018,cinti_solutions_2016,cabre_stable_2021}; examples of constant (non-zero) nonlocal mean curvature and the nonlocal Alexandrov soap bubble theorem \cite{ciraolo_rigidity_2018,cabre_curves_2018,cabre_near-sphere_2018,cabre_delaunay_2018,davila_nonlocal_2016,fall_constant_2018}; and Yau's conjecture \cite{caselli2024yaus}. For a nice survey of the current literature, see \cite{dipierro_nonlocal_2018}.

In analogy to the classical case, if \(E\) is \(s\)-stationary, that is stationary\textsuperscript{\ref{note1}} with respect to the \(s\)-perimeter, and \(\partial E\) is sufficiently smooth then \(E\) satisfies \begin{align*}
    \mathrm H_{s,E}(x) := \lim_{\varepsilon \to 0^+} \int_{\R^n \setminus B_\varepsilon(x)} \frac{\chi_{\R^n \setminus E}(y) - \chi_E (y) }{\vert x - y \vert^{n+s }} \dd y =0
\end{align*} for all \(x\in \partial E\), see \cite{figalli_isoperimetry_2015}. It is conventional to call \( \mathrm H_{s,E}\) the \(s\)-mean curvature of \(E\). Much of the literature focuses on the regularity of sets that minimise the \(s\)-perimeter and relatively little is known about the regularity of sets with zero of \(s\)-mean curvature. Indeed, although~\eqref{gJJA6X52} was proven in \cite{caffarelli_nonlocal_2010} for minimisers of the \(s\)-perimeter, a proof that sets of zero \(s\)-mean curvature satisfy the surface density estimate~\eqref{jR0GCsky} is not currently available in the literature. The main result of this note establishes this estimate.

\begin{thm} \thlabel{jG1dQXin}
Let \(s\in (0,1)\) and \(E \subset \R^n \) be a measurable set with locally finite perimeter satisfying \(\mathrm H_{s,E}=0 \) on \(\partial^\ast E\). Then \begin{align*}
    \operatorname{Per}(E;B_R) \geq C R^{n-1}  \text{ for all }R>0.
\end{align*} The constant \(C>0\) depends only on \(n\) and a lower bound for \(s\).
\end{thm}

The precise definition \(\mathrm H_{s,E}=0 \) on \(\partial^\ast E\) for a set of locally finite perimeter is given in Section 2. Several articles in the literature have explored related density estimates to~\thref{jG1dQXin} including \cite{cinti_quantitative_2019} and \cite{caselli2024yaus,caselli2024fractional}. In \cite{cinti_quantitative_2019}, the authors establish upper surface density estimates of the form \begin{align}
    \operatorname{Per} (E;B_R(x)) \leq C R^{n-1} \label{0YzpLLRv}
\end{align} for all \(x\in \partial^\ast E\), \(R>0\) for stationary, stable sets \(E\) with respect to the \(s\)-perimeter and other nonlocal perimeters with more general kernels. See also \cite{cabre_stable_2021} where \(\mathrm{BV}\)-estimates, analogous to~\eqref{0YzpLLRv}, are established for stable solutions to the fractional Allen-Cahn equation. It is interesting to note that in the classical case of stable surfaces with zero mean curvature,~\eqref{0YzpLLRv} is known only for \(n=2\) and is famously open for \(n\geq 3\). In \cite{caselli2024yaus,caselli2024fractional}, the main focus is on establishing a nonlocal Yau's conjecture, that is, in any closed Riemannian manifold there exist infinitely many sets of zero \(s\)-mean curvature. In particular, they prove the estimate~\eqref{0YzpLLRv} for sets with zero \(s\)-mean curvature in closed Riemannian manifolds. We should mention that, though most of the techniques present in this paper are also present in \cite{cinti_quantitative_2019,caselli2024yaus,caselli2024fractional}, particularly \cite{caselli2024yaus,caselli2024fractional}, the result \thref{jG1dQXin} is never explicitly stated nor is its connection to the fractional Sobolev inequality on hypersurfaces, see below. Note that~\eqref{0YzpLLRv} does not hold for sets with zero \(s\)-mean curvature that are not stable. Indeed, a counterexample is provided by an infinite, periodically arranged sequence of slabs\textemdash more specifically, the set \(\R^{n-1}\times \bigcup_{k\in \Z}(2k\varepsilon,(2k+1)\varepsilon)\) for any \(\varepsilon>0\). This set has zero \(s\)-mean curvature for all \(\varepsilon>0\), but is not stable for \(\varepsilon\) sufficiently small, see \cite[Remark 2.3]{cabre_stable_2021}.

It is natural to ask whether inequality~\eqref{gJJA6X52} holds for sets with zero \(s\)-mean curvature. To the best of the author’s knowledge, this remains unknown in full generality. The two counterexamples mentioned in the introduction for the classical case\textemdash the slab and the catenoid\textemdash do not extend to the nonlocal setting: the slab has constant positive \(s\)-mean curvature, and the nonlocal catenoid grows linearly at infinity \cite{davila_nonlocal_2018}, so it satisfies~\eqref{gJJA6X52} for some suitable constant \(C>0\). Similarly, the periodic sequence of slabs that served as a counterexample to~\eqref{0YzpLLRv} for non-stable sets satisfies~\eqref{gJJA6X52}, and therefore is also not a counterexample.

We conclude this remark by noting that~\eqref{gJJA6X52} does hold under a uniform perimeter bound\textemdash such as in the case of \(s\)-stationary stable sets, as shown in \cite{cinti_quantitative_2019}\textemdash and this will be further detailed in an upcoming paper \cite{cozzi_halfspace_inprep}.

A nice application of \thref{jG1dQXin} is that the fractional Sobolev inequality holds on hypersurfaces in \(\R^n\) with zero \(s\)-mean curvature. The classical Sobolev inequality in \(\R^n\) states that if \(1\leq p < n \) and \(p^\ast := \frac{np}{n-p}\) then \begin{align}
    \| u \|_{L^{p^\ast}(\R^n)} \leq C \| \nabla u \|_{L^p(\R^n)} \label{3U5YSnYf}
\end{align} for all \(u \in C^\infty_0(\R^n)\) with \(C>0\) depending only on \(n\) and \(p\). It is also interesting to ask if the Sobolev inequality can hold on hypersurfaces \(M^n \hookrightarrow \R^{n+1}\). It is easy to see that~\eqref{3U5YSnYf} (with \(\R^n\) replaced with \(M\) and \(\nabla = \nabla_M\) the gradient with respect to the induced metric from \(\R^{n+1}\)) cannot hold verbatim since if \(M\) is compact then \(u=1\in C^\infty_0(M)\); however, a Sobolev-type inequality does hold if one adds an extra \(L^p\) error term involving the mean curvature \(\mathrm H_M\): \begin{align}
    \| u \|_{L^{p^\ast}(M)} \leq C \big ( \| \nabla_M u \|_{L^p(M)} + \|\mathrm H_M u \|_{L^p(M) }\big ) \label{RMdxxBBd}
\end{align} for all \(u \in C^\infty_0(M)\). The inequality~\eqref{RMdxxBBd} is sometimes known as the Michael-Simon and Allard inequality after \cite{allard_first_1972} and \cite{michael_sobolev_1973}. Also, see \cite{brendle_isoperimetric_2021} where the optimal constant in~\eqref{RMdxxBBd} is obtained and \cite{cabre_universal_2022} for a simple proof of~\eqref{RMdxxBBd}. 

Naturally, one would like to know if~\eqref{RMdxxBBd} extends to the nonlocal case. If \(M=\R^n\), \(1\leq p < n/s\), and \(p^\ast = \frac{np}{n-sp}\) then it is well known that\begin{align*}
    \| u \|_{L^{p^\ast}(\R^n)} \leq C [u]_{W^{s,p}(\R^n )}
\end{align*} where \begin{align}
    [v]_{W^{s,p}(\R^n)} &= \bigg ( \int_{\R^n } \int_{\R^n}\frac{\vert v(x) - v(y) \vert^p }{\vert x - y \vert^{n+sp}} \dd y \dd x \bigg )^{\frac 1 p }, \label{JS7OQLcc}
\end{align}see \cite{di_nezza_hitchhikers_2012} and references therein. If \(E \subset \R^n\) is an open set with locally finite perimeter (or sufficiently smooth boundary) and \begin{align*}
      [v]_{W^{s,p}(\partial^\ast E)} &= \bigg ( \int_{\partial^\ast E } \int_{\partial^\ast E}\frac{\vert v(x) - v(y) \vert^p }{\vert x - y \vert^{n+sp}} \dd \mathcal H^{n-1}_y \dd \mathcal H^{n-1}_x \bigg )^{\frac 1 p }
\end{align*} then it is currently an open problem whether the inequality \begin{align}
     \| u \|_{L^{p^\ast}(\partial^\ast E)} \leq C \big (  [u]_{W^{s,p}(\partial^\ast E)} + \| \mathrm H_{s,E}u\|_{L^p(\partial^\ast E)} \big ) \label{yp5YLcrO}
\end{align} holds for all \(u\in C^\infty_0(\partial^\ast E)\). Recently, it was shown in \cite{cabre_fractional_2023} that~\eqref{yp5YLcrO} does hold provided that \(E\) is convex. Their argument relies on first establishing a pointwise lower bound on the \(s\)-mean curvature in terms of the perimeter. Unfortunately, as they point out in their paper, this pointwise inequality clearly cannot hold in the non-convex case, so the argument doesn't easily generalise to the case of arbitrary \(E\). Furthermore, in \cite[Proposition 5.2]{cabre_gradient_2019}, it was proven via an extension of an (unpublished) argument due to Brezis (see \cite[Proposition 15.5]{ponce_elliptic_2016} or \cite[Theorem 2.2.1]{bucur_nonlocal_2016} for this argument), that the boundary density estimate~\eqref{jR0GCsky} implies~\eqref{yp5YLcrO} without the term involving the \(s\)-mean curvature. Hence, via \cite[Proposition 5.2]{cabre_gradient_2019}, \thref{jG1dQXin} immediately implies that~\eqref{yp5YLcrO} holds for sets of zero \(s\)-mean curvature:

\begin{cor}
Let \(s,\alpha\in (0,1)\), \(1\leq p <n/ \alpha\), \(p^\ast = \frac{np}{n-\alpha p}\), and \(E \subset \R^{n+1} \) be a measurable set with locally finite perimeter satisfying \(\mathrm H_{s,E}=0 \) on \(\partial^\ast E\). Then the Sobolev inequality \begin{align*}
    \| u\|_{L^{p^\ast}(\partial^\ast E) } \leq C [u]_{W^{\alpha,p}(\partial^\ast E)}
\end{align*} holds for all \(u\in C^\infty_0(\partial^\ast E)\). The constant \(C>0\) depends only on \(n\), \(\alpha\), and \(s\). 
\end{cor}

\subsection{Organisation of paper}
The paper is organised as follows. In Section 2, we fix notation and give standard definitions of the classical theory of BV functions and minimal surfaces, and definitions from the theory of nonlocal minimal surfaces. In Section 3, we prove an interpolation inequality between restricted \(s\)-perimeter and the classical perimeter. In Section 4, we survey some results related to the Caffarelli-Silvestre extension in the context of nonlocal minimal surfaces that will be essential to the proof of \thref{jG1dQXin}. Finally, in Section 5, we give the proof of \thref{jG1dQXin}.

\section*{Acknowledgements}
Jack Thompson is supported by an Australian Government Research Training Program Scholarship. JT would like to thank the anonymous referee for their useful and interesting comments, and Serena Dipierro, Giovanni Giacomin, and Enrico Valdinoci for their valuable conversations and comments on the first draft of this note. He would also like to thank Jo{\~a}o Gon\c{c}alves da Silva for his suggestions and Tommaso Di Ubaldo for his deep insights.

\section{Definitions and notation}
First, we recall some notation and definitions of functions of bounded variation and the classical theory of minimal surfaces. Let \(\Omega \subset \R^n\) be an open set. Given a function \(u\in L^1(\Omega)\) the \emph{total variation} of \(u\) in \(\Omega\) is given by \begin{align*}
    \vert \nabla u \vert(\Omega) = \sup \bigg \{ \int_\Omega u \div \phi \dd x \text{ s.t. } \phi \in C^1_0(\Omega ; \R^n ), \| \phi \|_{L^\infty(\Omega ; \R^n)} \leq 1 \bigg \} . 
\end{align*} The space \(\operatorname{BV}(\Omega) \) is defined to be the set of \(u\in L^1(\Omega)\) such that \(\vert \nabla u \vert (\Omega) <+\infty \) and the space \(\operatorname{BV}_{\mathrm{loc}}(\Omega) \) is the set of functions in \(\operatorname{BV}(\Omega')\)  for all \(\Omega'\subset \subset \Omega\). Moreover, we say a measurable set \(E\subset \R^n\) has \emph{finite perimeter} in \(\Omega\) if \(\chi_E \in \operatorname{BV}(\Omega) \) where \(\chi_A\) denotes the characteristic function of \(A\). In this case, we define the \emph{perimeter of} \(E\) in \(\Omega\) by \begin{align*}
    \operatorname{Per}(E;\Omega) =  \vert \nabla \chi_E \vert(\Omega).
\end{align*} We also say \(E\) has \emph{locally finite perimeter} in \(\Omega\) if \(\chi_E\in \operatorname{BV}_{\mathrm{loc}}(\Omega)\) and simply \(E\) has locally finite perimeter if \(\chi_E\in \operatorname{BV}_{\mathrm{loc}}(\R^n)\). Finally, given \(E\) with finite perimeter in \(\Omega\), the distributional gradient of \(\chi_E\), i.e. \(\nabla \chi_E\), is a vector-valued Radon measure on \(\R^n\). Then we can define the \emph{reduced boundary} of \(E\), denoted \(\partial^\ast E \), as the set \begin{align*}
   \bigg \{ x\in \partial E \text{ s.t. } \vert \nabla \chi_E  \vert(B_r(x))>0 \text{ for all }r>0, \text{ and } \lim_{r\to 0^+} \frac{\nabla \chi_E (B_r(x)) }{\vert \nabla \chi_E  \vert(B_r(x)) } \text{ exists and is in } \Sph^{n-1}  \bigg \} .
\end{align*}

Now, let us turn to the nonlocal case. Let \(s\in (0,1)\) and \begin{align*}
    \mathcal I_s(A,B) = \int_A \int_B \frac{\dd y \dd x }{\vert x - y \vert^{n+s}} . 
\end{align*} The \(s\)-perimeter of a measurable set \(E\subset \R^n\) in \(\Omega\) is defined to be \begin{align*}
     \operatorname{Per}_s(E ; \Omega  ) = \mathcal I_s (E\cap \Omega , E^c \cap \Omega) +\mathcal I_s (E\cap \Omega , E^c \setminus \Omega)+\mathcal I_s (E\setminus \Omega, E^c\cap \Omega). 
\end{align*} We will also sometimes refer to \(\operatorname{Per}_s(E ; \Omega  )\) as the \emph{restricted \(s\)-perimeter}. Note, we also have that \begin{align}
    \operatorname{Per}_s(E ; \Omega  ) = \frac12\iint_{\mathcal Q(\Omega) } \frac{\vert \chi_E(x)-\chi_E(y) \vert}{\vert x - y \vert^{n+s} } \dd y \dd x \label{LDtSA29N}
\end{align} where we define \begin{align*}
    \mathcal Q(\Omega) = (\Omega^c \times \Omega^c)^c = (\Omega \times \Omega) \cup (\Omega \times \Omega^c) \cup (\Omega^c\times \Omega) \subset \R^{2n}.
\end{align*}

A measurable set \(E\) is a \emph{minimiser} of the \(s\)-perimeter in \(\Omega\) or \(s\)-minimal in \(\Omega\) if \begin{align*}
    \operatorname{Per}_s(E; \Omega) \leq \operatorname{Per}_s(F;\Omega)
\end{align*} for all measurable set \(F\) such that \(E\setminus\Omega = F\setminus \Omega\). Moreover, a measurable set \(E\) is stationary with respect to the \(s\)-perimeter in \(\Omega\) or \(s\)-stationary if \( \frac{\dd }{\dd t}\bigg \vert_{t=0} \operatorname{Per}_s(E_{t,T }; \Omega)\) exists and \begin{align}
    \frac{\dd }{\dd t}\bigg \vert_{t=0} \operatorname{Per}_s(E_{t,T }; \Omega)=0 \label{3MedX4AV}
\end{align} for all \(T\in C^\infty_0(\Omega; \R^n )\) where \begin{align*}
    E_{t,T} := \psi_T (E,t)
\end{align*} and \(\psi_T : \R^n \times (-\varepsilon , \varepsilon) \to \R^n\) satisfies \begin{align*}
    \begin{PDE}
\partial_t \psi_T &= T\circ \psi_T, &\text{in } \R^n \times (-\varepsilon , \varepsilon) \\
\psi_T &= \operatorname{Id}_{\R^n}, &\text{on } \R^n \times \{t=0\}.
    \end{PDE}
\end{align*} One can also define \(s\)-stationary stable sets, but we do not require this definition for the current paper,  see, for example, \cite[Section 6]{figalli_isoperimetry_2015}.

\section{An interpolation inequality}
In this section, we prove an interpolation inequality which will allow us to estimate the restricted \(s\)-perimeter with the restricted perimeter plus an arbitrarily small error. The precise statement is as follows.

\begin{thm} \thlabel{7qaGfSn3}
Let \(R>0\) and \(u\in \operatorname{BV}_{\mathrm{loc}}(\R^n) \cap L^\infty(\R^n)\). Then, for all \(\varepsilon\in (0,3^{- s } )\), \begin{align*}
    \iint_{\mathcal Q (B_R)} \frac{\vert u(x) - u(y)\vert }{\vert x - y \vert^{n+s}} \dd y \dd x \leq \frac C s  \bigg (\frac{1}{1-s} \| u\|_{L^1(B_{2\varepsilon^{-1/s}R})}^s\vert \nabla u\vert^{1-s}(B_{2\varepsilon^{- 1/s  }R})  + \varepsilon  R^{n-s} \| u\|_{L^\infty(\R^n) } \bigg ).
\end{align*} The constant \(C>0\) depends only on \(n\).
\end{thm}

For similar interpolation inequalities to~\thref{7qaGfSn3} and~\thref{60sDucQ0} below, see \cite[Equation 3.]{cinti_quantitative_2019} and \cite[Lemma 3.15]{caselli2024fractional}. A small, but important detail in \thref{7qaGfSn3} is the inclusion of the \(\varepsilon\) term which will be essential to the proof of \thref{jG1dQXin} since it will allow us to absorb an error term coming from \thref{7qaGfSn3} into the left-hand side of a chain of inequalities.

Taking \(u = \chi_E\) and \(u= \chi_{E^c}\) in \thref{7qaGfSn3}, we immediately obtain the following corollary. 

\begin{cor} \thlabel{60sDucQ0}
Let \(R>0\) and \(E \subset \R^n\) be a measurable set with locally finite perimeter. Then, for all \(\varepsilon\in(0,3^{- s}  ) \), \begin{align*}
    \operatorname{Per}_s(E; B_R) \leq  \frac C s  \bigg (\frac{1}{1-s} \min \{ \vert  E\cap B_{2\varepsilon^{-1/s}R}\vert ,\vert  B_{2\varepsilon^{-1/s}R} \setminus E\vert \}^{1-s} \cdot \operatorname{Per}^s(E;B_{2\varepsilon^{- 1/s  }R})  + \varepsilon  R^{n-s} \bigg ).
\end{align*}  The constant \(C>0\) depends only on \(n\).
\end{cor}

\begin{proof}[Proof of \thref{7qaGfSn3}]

First, assume that \(R=1\). Let \(\rho>0\) to be chosen later and write \begin{align*}
    \iint_{\mathcal Q(B_1)} \frac{\vert u(x) - u (y)\vert }{\vert x - y \vert^{n+s}} \dd y \dd x = I_\rho + J_\rho
\end{align*} where \begin{align*}
    I_\rho = \iint_{\mathcal Q(B_1)\cap \{ \vert x - y \vert <\rho \} } \frac{\vert u(x) - u (y)\vert }{\vert x - y \vert^{n+s}} \dd y \dd x \text{ and  }
    J_\rho = \iint_{\mathcal Q(B_1)\cap \{ \vert x - y \vert \geq \rho\} } \frac{\vert u(x) - u (y)\vert }{\vert x - y \vert^{n+s}} \dd y \dd x.
\end{align*}We will estimate \(I_\rho\) and \(J_\rho\) separately.

For \(I_\rho\), we have that \(\mathcal Q(B_1) \cap \{\vert x-y\vert<\rho\} \subset B_{1+\rho}\times B_{1+\rho}\), so \begin{align*}
    I_\rho \leq \int_{B_{1+\rho}} \int_{B_{1+\rho} } \frac{\vert u(x) - u (y)\vert }{\vert x - y \vert^{n+s}} \dd y \dd x. 
\end{align*} Furthermore, \begin{align*}
    \int_{B_{1+\rho}} \int_{B_{1+\rho} } \frac{\vert u(x) - u (y)\vert }{\vert x - y \vert^{n+s}} \dd y \dd x \leq \frac{C(n)}{s(1-s)} \| u\|_{L^1(B_{1+\rho})}^{1-s} \| \nabla u\|_{L^1(B_{1+\rho})}^s 
\end{align*} see \cite[Proposition 4.2]{brasco_characterisation_2021}, so \begin{align*}
    I_\rho \leq \frac{C}{s(1-s)} \| u\|_{L^1(B_{1+\rho})}^{1-s} \| \nabla u\|_{L^1(B_{1+\rho})}^s
\end{align*}

    For \(J_\rho\), \begin{align*}
        J_\rho \leq 2 \| u\|_{L^\infty(\R^n) } \iint_{\mathcal Q(B_1)\cap \{ \vert x - y \vert \geq \rho\} } \frac{\dd y \dd x }{\vert x - y \vert^{n+s}} .
    \end{align*} Assuming that \(\rho>3\), \( \mathcal Q(B_1)\cap \{ \vert x - y \vert \geq \rho\} \subset \big ( (B_1 \times B_1^c) \cup (B_1^c \times B_1)\big ) \cap \{ \vert x - y \vert \geq \rho\} \). Indeed, if there exists \((x,y) \in (B_1\times B_1)\cap \{ \vert x - y \vert \geq \rho\} \) then \(3<\rho \leq \vert x- y \vert \leq \vert x \vert + \vert y \vert <2\), a contradiction. Hence, \begin{align*}
         J_\rho \leq 4 \| u\|_{L^\infty(\R^n) } \iint_{ (B_1 \times B_1^c)\cap \{ \vert x - y \vert \geq \rho\} } \frac{\dd y \dd x }{\vert x - y \vert^{n+s}} . 
    \end{align*} Since \(\vert x-y\vert \geq \vert y \vert - \vert x \vert \geq \vert y \vert - 1 \) and \( (B_1 \times B_1^c)\cap \{ \vert x - y \vert \geq \rho\} \subset B_1 \times B_{\rho -1 }^c \), we find that \begin{align*}
         J_\rho &\leq C  \| u\|_{L^\infty(\R^n) } \int_{B_1} \int_{\R^n \setminus B_{\rho -1}} \frac{\dd y \dd x }{(\vert y \vert - 1  )^{n+s}} .
    \end{align*} Since \(\rho>3\), we obtain \begin{align*}
         J_\rho \leq C  \| u\|_{L^\infty(\R^n) } \int_{B_1} \int_{\R^n \setminus B_{\rho -1}} \frac{\dd y \dd x }{\vert y \vert^{n+s}} \leq \frac{ C(\rho -1)^{-s}} s  \| u\|_{L^\infty(\R^n) } \leq \frac{ C\rho ^{-s}} s  \| u\|_{L^\infty(\R^n) }
    \end{align*} with \(C\) depending only on \(n\).

   Thus, combining the above estimates, we have \begin{align*}
        \iint_{\mathcal Q(B_1)} \frac{\vert u(x) - u (y)\vert }{\vert x - y \vert^{n+s}} \dd y \dd x &\leq \frac C s \bigg ( \frac1{1-s} \| u\|_{L^1(B_{1+\rho})}^{1-s} \vert \nabla u\vert^{s}(B_{1+\rho}) + \rho ^{-s}  \| u\|_{L^\infty(\R^n) }\bigg ) \\
        &\leq  \frac C s \bigg ( \frac1{1-s} \| u\|_{L^1(B_{2\rho})}^{1-s} \vert \nabla u\vert^{s}(B_{2\rho}) + \rho ^{-s}  \| u\|_{L^\infty(\R^n) }\bigg ) 
    \end{align*} for all \(\rho>3\). Choosing \(\rho = \varepsilon^{- \frac 1 s }\), we prove the result. 

    Finally, to obtain the case for general \(R>0\), let \(u_R(x)= u(Rx)\). Then \begin{align*}
        \iint_{\mathcal Q(B_R)} \frac{\vert u(x) - u (y)\vert }{\vert x - y \vert^{n+s}} \dd y \dd x &= R^{n-s}\iint_{\mathcal Q(B_1)} \frac{\vert u_R(x) - u_R (y)\vert }{\vert x - y \vert^{n+s}} \dd y \dd x \\
        &\leq \frac {CR^{n-s}} s  \bigg (\frac1{1-s} \| u_R\|_{L^1(B_{2\varepsilon^{- 1/s }})}^{1-s} \vert \nabla u_R\vert^s(B_{2\varepsilon^{- 1/s  }}) + \varepsilon  \| u_R\|_{L^\infty(\R^n) } \bigg ) \\
        &= \frac {CR^{n-s}} s  \bigg (\frac{R^{s-n}}{1-s} \| u\|_{L^1(B_{2\varepsilon^{-1/s}R})}^{1-s}\vert \nabla u\vert^{s}(B_{2\varepsilon^{- 1/s  }R})  + \varepsilon  \| u\|_{L^\infty(\R^n) } \bigg ).
    \end{align*} 
\end{proof}

\section{The extension problem}

In this section, we record several results regarding the Caffarelli-Silvestre extension, in the context of the theory of nonlocal minimal surfaces, that we require in the proof of~\thref{jG1dQXin}. Throughout this section, we will denote points in \(\R^{n+1}\) with capital letters \(X\), \(Y\), etc and write \(X=(x,x_{n+1})=(x_1,\dots,x_{n+1})\) where \(x=(x_1,\dots,x_n) \in \R^n \) (and analogously \(Y=(y,y_{n+1})\), etc). Moreover, let \begin{align*}
    \R^{n+1}_+ &= \{ X \in \R^{n+1} \text{ s.t. } x_{n+1}>0\}, \\
    \tilde B_R(X) &= \{ Y \in \R^{n+1} \text{ s.t. } \vert Y-X\vert <R\}, \\
    \tilde B_R^+(X) &=   \tilde B_R(X) \cap \R^{n+1}_+,  \\ 
    \tilde B_R &= \tilde B_R(0), \text{ and} \\
    \tilde B_R^+ &= \tilde B_R^+(0) .
\end{align*}

Let \(s\in (0,1)\), \(E\subset \R^n\) be a measurable set, and \begin{align*}
    \tilde \chi _E(x) &= \chi_{\R^n \setminus E}(x) - \chi_E(x)  \qquad x\in \R^n
\end{align*} where  \(\chi_E\) is the characteristic function of \(E\). Now we define \(U_E: \R^{n+1}_+ \to \R \) \begin{align*}
    U_E(X) &= \int_{\R^n} P_{s/2}(X,y) \tilde \chi _E(y) \dd y, \qquad \text{for all } X\in \R^{n+1}_+
\end{align*} where \begin{align}
   P_{s/2}(X,y) =a(n,s) \frac{x_{n+1}^s}{\big ( \vert x-y\vert^2 + x_{n+1}^2 \big )^{\frac{n+s}2} }, \qquad  a(n,s) = \frac{\Gamma \big (\frac {n+s} 2 \big )  }{\pi^{\frac n2} \Gamma \big ( \frac s 2 \big ) }. \label{kzYTNNBU}
\end{align} Since \(\vert  \tilde \chi _E(x) \vert \leq 1 \) and \(\int_{\R^n}P_{s/2}(X,y) \dd y =1 \) for all \(X\in \R^{n+1}_+\), see \cite[Remark 10.2]{garofalo_fractional_2019}, \(U_E\) is well-defined, and, in fact, is a smooth function. Moreover, \(U_E\) satisfies the degenerate PDE \begin{align*}
    \begin{PDE}
\div \big ( x_{n+1}^{1-s} \nabla V \big ) &= 0, &\text{in } \R^{n+1}_+ \\
V &= \tilde \chi_E, &\text{on } \R^n \times \{0\}. 
    \end{PDE} 
\end{align*} where \(\nabla\) and \(\div\) denote the gradient and divergence in \(\R^{n+1}\) respectively. Furthermore, \(V=\tilde \chi_E\) on \(\R^n\times \{0\}\) is understood in the trace sense with respect to the unique trace operator from the weighted Sobolev space \(H^1(\R^{n+1}_+;x_{n+1}^{1-s}\dd X)\) to \(L^2(\R^n)\), see \cite[Theorem 9.1]{leoni_first_2023}. This PDE was first studied in the context of the fractional Laplacian in \cite{caffarelli_extension_2007} and in the context of nonlocal minimal surfaces in \cite{caffarelli_nonlocal_2010}. In particular, it was proven that if \(E\) is a minimiser of the \(s\)-perimeter then the function \begin{align}
   \Phi_{E,x}(R) := \frac 1 {R^{n-s}} \int_{\tilde B_R^+((x,0))} y_{n+1}^{1-s} \vert \nabla U_E(Y) \vert^2 \dd Y \label{HvcZ5aGS}
\end{align} is monotone increasing. This function plays an identical role in the theory as the function \( R^{1-n} \operatorname{Per}(E;B_R(x)) \) plays in the classical setting. See also \cite{garofalo_fractional_2019} and references therein for more details. Recently, in \cite{caselli2024fractional}, it was shown that the monotonicity of~\eqref{HvcZ5aGS} also holds for sets that are stationary with respect to the \(s\)-perimeter: 

\begin{prop}[{\cite[Theorem 3.4]{caselli2024fractional}}] \thlabel{uUBAsC54}
Let \(s\in (0,1)\) and \(E\subset \R^n\) be a measurable set that is stationary with respect to the \(s\)-perimeter. Then \(\Phi_{E,x}\) is monotone increasing for all \(x\in \partial^\ast E\).
\end{prop}

Moreover, in \cite{caffarelli_nonlocal_2010} the following ``extension--trace" energy estimate was proven which gives a control on \(\Phi_{E,x}\) in terms of the \(s\)-perimeter of \(E\) in a ball.

\begin{prop}[{\cite[Propostion 7.1(a)]{caffarelli_nonlocal_2010}}] \thlabel{Qu1H6WRf}
    Let \(s\in (0,1)\) and \(E\subset \R^n\) have locally finite perimeter. Then  \begin{align*}
        \Phi_{E,x}(R) \leq \frac {Cs^{-2}} {R^{n-s}} \operatorname{Per}_s(E; B_{2R}) 
    \end{align*} for all \(R>0\) and \(x\in \partial^\ast E\). The constant \(C>0\) depends only on \(n\).
\end{prop}

The proof of \thref{Qu1H6WRf} is given \cite{caffarelli_nonlocal_2010} (applied to \(\chi_{\R^n \setminus E} - \chi_E\) and given again with more details in \cite[Lemma 3.13]{caselli2024fractional}. We will sketch the proof given \cite[Lemma 3.13]{caselli2024fractional} in our particular case, just to be explicit about how the constant \(C(n)s^{-2}\) was obtained.

\begin{proof}
By translating and rescaling, it is sufficient to consider the case \(R=1\) and \(x=0\). We follow the argument given in \cite[Lemma 3.13]{caselli2024fractional} applied to \(u= \chi_{\R^n\setminus E}-\chi_E - \big ( \vert B_2\setminus E\vert - \vert E\cap B_2\vert  \big )\). Note, we have chosen \(u\) such that \(\int_{B_2} u \dd x = 0\). Let \(\zeta \in C^\infty_0(\R^n)\) be such that \(\zeta =1 \) in \(B_{3/2}\), \(\zeta =0 \) in \(\R^n \setminus B_2\), and \(0\leq \zeta \leq 1\). Define \(u_1=\zeta u\), \(u_2= (1-\zeta)u\), and \(U_1\) and \(U_2\) the Caffarelli-Silvestre extensions of \(u_1\) and \(u_2\) respectively. We observe that our choice of normalisation constant implies \begin{align}
    \int_{\R^{n+1}_+} x_{n+1}^{1-s} \vert \nabla U_1\vert^2 \dd X = \frac{s \Gamma \big ( \frac{n+s}2 \big )}{2\pi^{\frac n2} \Gamma \big ( \frac s2 \big ) } [u_1]_{H^{s/2}(\R^n)}^2 \label{ytKVzdar}
\end{align} where \([\cdot ]_{H^{s/2}(\R^n)}\) is given by~\eqref{JS7OQLcc}. Note that the constant in~\eqref{ytKVzdar} is comparable to \(s^2\). Indeed, \(P_{s/2}\) is the same as the one in \cite{garofalo_fractional_2019} (after replacing \(s\) with \(s/2\)), so from \cite[Proposition 10.1]{garofalo_fractional_2019} (see also \cite[Remark 10.5]{garofalo_fractional_2019}), we have via an integration by parts \begin{align*}
    \int_{\R^{n+1}_+} x_{n+1}^{1-s} \vert \nabla U_1\vert^2 \dd X &= - \int_{\R^n} u_1(x) \lim_{t\to 0^+} t^{1-s} \partial_{n+1}U_1(x,t) \dd x \\
    &= \frac{s \Gamma \big ( \frac{n+s}2 \big )}{\pi^{\frac n2} \Gamma \big ( \frac s2 \big ) }  \int_{\R^n}\int_{\R^n} \frac{u_1(x)(u_1(x)-u_1(y))}{\vert x- y \vert^{n+s}} \dd y \dd x \\
    &= \frac{s \Gamma \big ( \frac{n+s}2 \big )}{2\pi^{\frac n2} \Gamma \big ( \frac s2 \big ) }  [u_1]_{H^{s/2}(\R^n)}^2. 
\end{align*} Moreover, in \cite{caselli2024fractional} they apply the fractional Poincaré-Wirtinger inequality \begin{align}
    \| u  \|_{L^2(B_2)}^2 \leq C(n) \int_{B_2}\int_{B_2} \frac{\vert u(x)-u(y)\vert^2}{\vert x-y \vert^{n+s}} \dd y \dd x , \label{rmU2m8rh}
\end{align} see for example \cite[Equation 4.2]{mingione_singular_2003}. 
All together, the argument in \cite{caselli2024fractional} and~\eqref{LDtSA29N} gives \begin{align*}
    \int_{\R^{n+1}_+} x_{n+1}^{1-s}\vert \nabla U_1 \vert^2 \dd X \leq C(n)s^2 \operatorname{Per}_s(E;B_2) \leq C(n) \operatorname{Per}_s(E;B_2).
\end{align*}

Furthermore, as in the proof of \cite[Lemma 3.13]{caselli2024fractional}, \begin{align*}
    x_{n+1}^{1-s} \vert \nabla U_2(X) \vert \leq C(n) \int_{\R^n} \frac{\vert u_2(y) \vert }{(1+\vert y \vert^2 )^{\frac{n+s}2}} \dd y\leq C(n) \int_{\R^n} \frac{\vert u(y) \vert }{(1+\vert y \vert^2 )^{\frac{n+s}2}} \dd y \qquad \text{for all }x\in \tilde B_1^+.
\end{align*} Then, by H\"older's inequality,  \begin{align*}
    \int_{\R^n} \frac{\vert u(y) \vert }{(1+\vert y \vert^2 )^{\frac{n+s}2}} \dd y &\leq \bigg ( \int_{\R^n} \frac{\dd y }{(1+\vert y \vert^2 )^{\frac{n+s}2}} \bigg)^{\frac 12 } \bigg ( \int_{\R^n} \frac{\vert u(y) \vert^2 }{(1+\vert y \vert^2 )^{\frac{n+s}2}} \dd y \bigg)^{\frac 12 } .
\end{align*} Since \begin{align*}
    \int_{\R^n} \frac{\dd y  }{(1+\vert y \vert^2 )^{\frac{n+s}2}} \dd y &= \omega_{n-1} \int_0^{+\infty} \frac{t^{n-1}\dd t}{(1+t^2)^{\frac{n+s}2}} =\frac{\omega_{n-1} \Gamma \left(\frac{n}{2}\right) \Gamma \left(\frac{s}{2}\right)}{2 \Gamma \left(\frac{n+s}{2}\right)}, 
\end{align*} we obtain \begin{align*}
     \int_{\R^n} \frac{\vert u(y) \vert }{(1+\vert y \vert^2 )^{\frac{n+s}2}} \dd y &\leq C(n)s^{-\frac 12}\bigg ( \int_{\R^n} \frac{\vert u(y) \vert^2 }{(1+\vert y \vert^2 )^{\frac{n+s}2}} \dd y \bigg)^{\frac 12 } . 
\end{align*} Moreover, using that \(\int_{B_2} u \dd y =0\), \begin{align*}
    \int_{\R^n} \frac{\vert u(x) \vert^2 }{(1+\vert x \vert^2 )^{\frac{n+s}2}} \dd x  &= \frac 1 {\vert B_2\vert }\int_{\R^n}\int_{B_2} \frac{(u(x)-u(y))^2}{(1+\vert x \vert^2 )^{\frac{n+s}2}} \dd y \dd x-\frac 1 {\vert B_2\vert }\int_{\R^n}\int_{B_2} \frac{\vert u(y)\vert ^2}{(1+\vert x \vert^2 )^{\frac{n+s}2}} \dd y \dd x \\
    &\leq C(n) \iint_{\mathcal Q(B_2)}  \frac{\vert u(x)-u(y)\vert^2}{\vert x-y \vert^{n+s}} \dd y \dd x \\
    &\leq C(n) \operatorname{Per}_s(E;B_2)
\end{align*} Thus, we obtain \begin{align*}
    \int_{\tilde B_1^+} x_{n+1}^{1-s}\vert \nabla U_2 \vert^2 \dd X &\leq C(n) \bigg (  \int_{\R^n} \frac{\vert u(y) \vert }{(1+\vert y \vert^2 )^{\frac{n+s}2}} \dd y \bigg ) \int_{\tilde B_1^+} \vert \nabla U_2\vert\dd X \\
    &\leq C(n)s^{-1/2} \bigg ( \int_{\R^n} \frac{\vert u(y) \vert^2 }{(1+\vert y \vert^2 )^{\frac{n+s}2}} \dd y \bigg )^{1/2}\bigg(\int_{\tilde B_1^+} x_{n+1}^{s-1} \dd X\bigg)^{\frac 12 }  \bigg(\int_{\tilde B_1^+} x_{n+1}^{1-s}\vert \nabla U_2\vert^2 \dd X \bigg)^{\frac 12 } \\
    &\leq C(n) s^{-1}\bigg(\int_{\tilde B_1^+} x_{n+1}^{1-s}\vert \nabla U_2\vert^2 \dd X \bigg)^{\frac 12 }\operatorname{Per}_s^{1/2}(E;B_2),
\end{align*} so \begin{align*}
    \int_{\tilde B_1^+} x_{n+1}^{1-s}\vert \nabla U_2 \vert^2 \dd X \leq C(n)s^{-2} \operatorname{Per}_s(E;B_2).
\end{align*}
Thus, we obtain \begin{align*}
     \int_{\tilde B^+_1} x_{n+1}^{1-s}\vert \nabla U_E \vert^2 \dd X &\leq C \bigg (  \int_{\R^n_+} x_{n+1}^{1-s}\vert \nabla U_1 \vert^2 \dd X+\int_{\tilde B^+_1} x_{n+1}^{1-s}\vert \nabla U_2 \vert^2 \dd X \bigg ) \\
     &\leq C(n) \bigg ( 1 + \frac 1{s^2} \bigg ) \operatorname{Per}_s(E;B_2) \\
     &\leq \frac{C(n)} {s^2} \operatorname{Per}_s(E;B_2).
\end{align*}
\end{proof}

Furthermore, we have the following proposition.

\begin{prop} \thlabel{GeFXLo81}
Let \(E\subset \R^n\) be measurable with locally finite perimeter. Then, for all \(x\in \partial^\ast E\), \begin{align}
    \lim_{R\to 0^+} \Phi_{E,x}(R) &=\Phi_{\R^n_+,0}(1)  \label{bpwPOxAA}
\end{align} where \begin{align*}
    \Phi_{\R^n_+,0}(1)  = \frac{2 \pi^{\frac n 2 - 1 } \Gamma \big ( \frac{s+1}2 \big )\Gamma \big ( \frac{1-s}2 \big ) }{\Gamma \big ( \frac s2 \big )\Gamma \big ( \frac{n-s}2+1 \big )}>0.
\end{align*} In particular, \begin{align*}
    \lim_{R\to 0^+} \Phi_{E,x}(R) \geq \frac {Cs} {1-s} 
\end{align*} with \(C\) depending only on \(n\).
\end{prop}

This proposition follows from \thref{aT0zFAPm} and \thref{PFeFnfBd} below, and is a nonlocal analogue of the limit \begin{align*}
    \lim_{R\to 0^+}\frac{\operatorname{Per}(E;B_R(x)) }{R^{n-1}} = \operatorname{Per}(\R^n_+;B_1) 
\end{align*} where \(E\) is a set of locally finite perimeter and \(x\in \partial^\ast E\), see \cite[Equation (15.9) in Corollary 15.8]{maggi_sets_2012}. For \(s\)-minimisers, \thref{GeFXLo81} follows from \cite[Theorem 9.1]{caffarelli_nonlocal_2010}. The proof given there takes advantage of estimates obtained via comparing the \(s\)-perimeter of \(E\) with a competitor (when they apply Theorem 3.3 in the same paper), so does not readily extend to the \(s\)-stationary case.


\begin{lem} \thlabel{aT0zFAPm}
     Let \(E_k\subset \R^n\) be a sequence of Borel sets that converge to a Borel set \(E\subset \R^n\) in \(L^1_{\mathrm{loc}}(\R^n)\), and, for any ball \(B\subset \R^n\), \( \operatorname{Per}(E_k\setminus E;B)\) and \( \operatorname{Per}(E\setminus E_k;B)\) are bounded uniformly in \(k\). Then, \(\nabla U_{E_k}\to\nabla U_E\) in \(L^2_{\mathrm{loc}}(\R^{n+1};x_{n+1}^{1-s}\dd X)\).
\end{lem}

\begin{proof}
    We will show that \begin{align*}
         \lim_{k\to +\infty }\int_{\tilde B_R^+((x,0))} y_{n+1}^{1-s} \vert \nabla (U_{E_k}-U_E)(Y)\vert^2\dd Y=0.
     \end{align*} for all \(x\in \R^n\) and \(R>0\). Without loss of generality, we may assume \(x=0\) and \(R=1\). Fix \(\varepsilon>0\) sufficiently small and let \(v_k:=\tilde \chi_{E_k}-\tilde \chi_E\). From \cite[Proposition 7.1(i)]{caffarelli_nonlocal_2010} and \thref{7qaGfSn3}, it follows that \begin{align*}
         \int_{\tilde B_1^+} y_{n+1}^{1-s} \vert \nabla (U_{E_k}-U_E)(Y)\vert^2\dd Y &\leq  C\iint_{\mathcal Q (B_2)} \frac{\vert v_k(x) - v_k(y)\vert }{\vert x - y \vert^{n+s}} \dd y \dd x \\
         &\leq C  \big ( \| v_k\|_{L^1(B_{4\varepsilon^{-1/s}})}^s\vert \nabla v_k\vert^{1-s}(B_{4\varepsilon^{- 1/s  }})  + \varepsilon  \| v_k\|_{L^\infty(\R^n) } \bigg ).
     \end{align*} Since \(v_k = 2 \chi_{E_k\setminus E}-2\chi_{E\setminus E_k}\), we have \begin{align*}
         \| v_k\|_{L^1(B_{4\varepsilon^{-1/s}})} &\leq 2 \vert (E_k \triangle E) \cap B_{4\varepsilon^{- 1/s  }} \vert \\
         \vert  \nabla v_k\vert(B_{4\varepsilon^{- 1/s  }})&\leq 2 \operatorname{Per}(E_k\setminus E;B_{4\varepsilon^{- 1/s  }})+ 2\operatorname{Per}(E\setminus E_k;B_{4\varepsilon^{- 1/s  }}) \leq C_\varepsilon\\
          \| v_k\|_{L^\infty(\R^n) } \leq 4
     \end{align*} where \(C_\varepsilon>0\) is a constant that may depend on \(\varepsilon>0\). Hence, \begin{align*}
           \int_{\tilde B_1^+} x_{n+1}^{1-s} \vert \nabla (U_{E_k}-U_E)\vert^2\dd X &\leq  C_\varepsilon  \vert (E_k \triangle E) \cap B_{4\varepsilon^{- 1/s  }} \vert^s + C \varepsilon 
     \end{align*} where \(C>0\) does not depend on \(\varepsilon\). Thus, \begin{align*}
         \limsup_{k\to +\infty }\int_{\tilde B_1^+} x_{n+1}^{1-s} \vert \nabla (U_{E_k}-U_E)\vert^2\dd X &\leq  C \varepsilon .
     \end{align*} Since this holds for all \(\varepsilon>0\) sufficiently small, we obtain \begin{align*}
         \lim_{k\to +\infty }\int_{\tilde B_1^+} x_{n+1}^{1-s} \vert \nabla (U_{E_k}-U_E)\vert^2\dd X=0.
     \end{align*}
\end{proof}

In the following lemma, we compute the value of \(\Phi_{\R^n_+,0}(1) \) explicitly. 

\begin{lem} \thlabel{PFeFnfBd}
    We have that \begin{align*}
        \Phi_{\R^n_+,0}(1) =  \frac{2 \pi^{\frac n 2 - 1 } \Gamma \big ( \frac{s+1}2 \big )\Gamma \big ( \frac{1-s}2 \big ) }{\Gamma \big ( \frac s2 \big )\Gamma \big ( \frac{n-s}2+1 \big )}.
    \end{align*}
\end{lem}

\begin{proof}
First, we claim that \begin{align}
    U_{\R^n_+} (X) &=  -\frac{2\Gamma \big ( \frac{s+1}2\big )}{\pi^{\frac12}\Gamma \big ( \frac s2\big )} h \bigg ( \frac{x_n}{x_{n+1}}\bigg ) \label{D9M0UGx2}
\end{align} where \begin{align*}
   h(\tau )&= \int_0^{\tau } \frac {\dd t } { \big ( t^2 +1 \big )^{\frac {1+s} 2} } .
\end{align*} Indeed, if \(X_\ast := (x',-x_n,x_{n+1})\) with \(x'=(x_1,\dots, x_{n-1})\in \R^{n-1}\) then the change of variables \(y \to (y',-y_n)\) in the second integral below gives \begin{align}
    U_{\R^n_+}(X) &= -a(n,s)  x_{n+1}^s\bigg ( \int_{\mathbb R^n_+} \frac1{\big ( \vert x - y \vert^2 +x_{n+1}^2 \big )^{ \frac{n+s}2 }} \dd y-\int_{\mathbb R^n_-} \frac1{\big ( \vert x - y \vert^2 +x_{n+1}^2 \big )^{ \frac{n+s}2 }} \dd y \bigg ) \nonumber  \\
    &= -a(n,s) x_{ n+1}^s \big ( f(X)-f(X_\ast) \big ) \label{VK6vmGJS}
\end{align} where \begin{align*}
    f(X) := \int_{\mathbb R^n_+} \frac{\dd y}{\big ( \vert x - y \vert^2 +x_{n+1}^2 \big )^{ \frac{n+s}2 }}.
\end{align*} Observe that, via a translation in the first \(n-1\) coordinates, we have \begin{align*}
    f(X) &= \int_0^{+\infty}\int_{\mathbb R^{n-1}} \frac{\dd y'\dd t }{\big ( \vert y '\vert^2 +  ( x_n - t )^2 +x_{n+1}^2 \big )^{ \frac{n+s}2 }}.
\end{align*} Next, the rescaling \(y'\to y'\sqrt{( x_n - t )^2 +x_{n+1}^2}\) in the inner integral gives \begin{align*}
    f(X) &= \bigg ( \int_{\mathbb R^{n-1}} \frac{\dd y' }{\big ( \vert y '\vert^2 + 1 \big )^{ \frac{n+s}2 }}\bigg )   \int_0^{+\infty} \frac{\dd t }{\big ( ( x_n - t )^2 +x_{n+1}^2 \big )^{ \frac{1+s}2 }}
\end{align*} An application of polar coordinates gives \begin{align}
    \int_{\mathbb R^{n-1}} \frac{\dd y' }{\big ( \vert y '\vert^2 + 1 \big )^{ \frac{n+s}2 }} = \frac{\pi^{\frac{n-1}2 \Gamma \big ( \frac{s+1}2\big ) }}{\Gamma \big (  \frac{n+s}2\big ) }, \label{4HylpNJ5}
\end{align} see \cite[Proposition 4.1]{garofalo_fractional_2019} for the explicit computation. Furthermore, the change of variables \(t \to x_{n+1} t \) gives \begin{align}
     \int_0^{+\infty} \frac{\dd t }{\big ( ( x_n - t)^2 +x_{n+1}^2 \big )^{ \frac{1+s}2 }} &=  x_{n+1}\int_0^{+\infty} \frac{\dd t }{\big ( ( x_n - x_{n+1} t )^2 +x_{n+1}^2 \big )^{ \frac{1+s}2 }} \nonumber \\
     &=  x_{n+1}^{-s} \int_0^{+\infty} \frac{\dd t }{\big ( ( x_n/x_{n+1} -  t )^2 +1 \big )^{ \frac{1+s}2 }}.  \label{5yHmdraj}
\end{align} If we define \begin{align*}
   g(\tau) = \int_0^{+\infty} \frac {\dd t } { \big ( (\tau -t  )^2 +1 \big )^{\frac {1+s} 2} }-\int_0^{+\infty} \frac {\dd t } { \big ( (\tau +t  )^2 +1 \big )^{\frac {1+s} 2} }
\end{align*} and \begin{align*}
    \tilde a(n,s) = 2 a(n,s) \frac{\pi^{\frac{n-1}2 \Gamma \big ( \frac{s+1}2\big ) }}{\Gamma \big (  \frac{n+s}2\big ) } = \frac{2\Gamma \big ( \frac{s+1}2\big )}{\pi^{\frac12}\Gamma \big ( \frac s2\big )}
\end{align*} then it follows from~\eqref{VK6vmGJS},~\eqref{4HylpNJ5}, and~\eqref{5yHmdraj} that \begin{align*}
  U_{\R^n_+}(X) &=-\frac12 \tilde a(n,s) \bigg [\int_0^{+\infty} \frac{\dd t }{\big ( ( x_n/x_{n+1} -  t )^2 +1 \big )^{ \frac{1+s}2 }}-\int_0^{+\infty} \frac{\dd t }{\big ( ( x_n/x_{n+1} +  t )^2 +1 \big )^{ \frac{1+s}2 }} \bigg ] \\
  &= -\frac12\tilde a(n,s)  g \bigg ( \frac{x_n}{x_{n+1}} \bigg ) .
\end{align*} Now, \begin{align*}
     g(\tau) &= \int_{-\tau}^{+\infty} \frac {\dd t } { \big ( t^2 +1 \big )^{\frac {1+s} 2} }-\int_{\tau}^{+\infty} \frac {\dd t } { \big ( t^2 +1 \big )^{\frac {1+s} 2} } \\
     &=  \int_{-\tau}^{\tau } \frac {\dd t } { \big ( t^2 +1 \big )^{\frac {1+s} 2} } \\
     &= 2 h(\tau)
\end{align*} which proves~\eqref{D9M0UGx2}. 

From~\eqref{D9M0UGx2}, it follows that\begin{align*}
    \nabla U_{\R^n_+}(X) &= - \tilde a (n,s) \bigg (0,\dots, 0 , \frac 1{x_{n+1}}h '\bigg ( \frac{x_n}{x_{n+1}}\bigg ), - \frac{x_n}{x_{n+1}^2} h '\bigg ( \frac{x_n}{x_{n+1}}\bigg ) \bigg ) \\
    &=  - \tilde a (n,s) \bigg (0,\dots, 0 , \frac{x_{n+1}^s }{\big ( x_n^2 + x_{n+1}^2 \big )^{\frac {1+s}2 } }, - \frac{x_nx_{n+1}^{s-1} }{\big ( x_n^2 + x_{n+1}^2 \big )^{\frac {1+s}2 } }  \bigg ) , 
\end{align*} so \begin{align*}
    x_{n+1}^{1-s} \vert \nabla U_E\vert^2 &= \tilde a (n,s)^2 \frac{x_{n+1}^{s-1} }{\big ( x_n^2 + x_{n+1}^2 \big )^s}
\end{align*} Using the notation, \(B^{2,+}_1 := \{ (x_n,x_{n+1} ) \in \R^2 \text{ s.t. } x_n^2+x_{n+1}^2<1, x_{n+1}>0 \}\), we have \begin{align*}
    \tilde B_1^+ &= \big \{ X \in \R^{n+1} \text{ s.t. } (x_n,x_{n+1} ) \in B^{2,+}_1, x' \in B^{n-1}_{\sqrt{1-x_n^2-x_{n+1}^2}} \big \},
\end{align*} so we obtain \begin{align*}
    \int_{\tilde B_1^+} x_{n+1}^{1-s} \vert \nabla U_{\R^n_+} \vert^2 \dd X &= \tilde a (n,s)^2 \int_{B^{2,+}_1}\int_{ B^{n-1}_{\sqrt{1-x_n^2-x_{n+1}^2}}} \frac{x_{n+1}^{s-1} }{\big ( x_n^2 + x_{n+1}^2 \big )^s} \dd x'\dd x_n\dd x_{n+1} \\
    &= \tilde a (n,s)^2 \omega_{n-1} \int_{B^{2,+}_1} \frac{x_{n+1}^{s-1}\big (1- x_n^2 - x_{n+1}^2 \big )^{\frac{n-1}2} }{\big ( x_n^2 + x_{n+1}^2 \big )^s} \dd x_n\dd x_{n+1}.
\end{align*} Hence, polar coordinates \((x_n,x_{n+1})=(r\cos \theta, r\sin \theta)\) give \begin{align*}
      \int_{\tilde B_1^+} x_{n+1}^{1-s} \vert \nabla U_{\R^n_+} \vert^2 \dd X &=\tilde a (n,s)^2 \omega_{n-1} \bigg ( \int_0^1 r^{-s} (1-r^2)^{\frac{n-1} 2 } \dd r \bigg ) \bigg (\int_0^{\pi } (\sin \theta)^{s-1} \dd \theta \bigg ) .
\end{align*} One can show that \begin{align*}
\int_0^1  r^{-s} \big (1-r^2 \big )^{ \frac{n-1}2 }     \dd r =\frac{\Gamma \left(\frac{n+1}{2}\right) \Gamma \left(\frac{1-s}{2}\right)}{2 \Gamma \left(\frac {n-s} 2 + 1 \right)} \text{ and }  \int_0^\pi  (\sin \theta )^{s-1}  \dd \theta  = \frac{\pi^{\frac 12 }  \Gamma \left(\frac{s}{2}\right)}{\Gamma \left(\frac{s+1}{2}\right)}
\end{align*} Thus, \begin{align*}
    \int_{\tilde B_1^+} x_{n+1}^{1-s} \vert \nabla U_{\R^n_+} \vert^2 \dd X &= \tilde a (n,s)^2 \omega_{n-1} \bigg ( \frac{\Gamma \left(\frac{n+1}{2}\right) \Gamma \left(\frac{1-s}{2}\right)}{2 \Gamma \left(\frac {n-s} 2 + 1 \right)} \bigg ) \bigg ( \frac{\pi^{\frac 12 }  \Gamma \left(\frac{s}{2}\right)}{\Gamma \left(\frac{s+1}{2}\right)}\bigg ) \\
    &= \frac{2 \pi^{\frac n 2 - 1 } \Gamma \big ( \frac{s+1}2 \big )\Gamma \big ( \frac{1-s}2 \big ) }{\Gamma \big ( \frac s2 \big )\Gamma \big ( \frac{n-s}2+1 \big )}
\end{align*} as required. 
\end{proof}

We now give the proof of \thref{GeFXLo81}.

\begin{proof}[Proof of \thref{GeFXLo81}]
    By translating and rescaling, we have that \begin{align*}
        \Phi_{E,x}(R)= \Phi_{R^{-1}(E-x),0}(1).
    \end{align*} Let \(\{R_k\}\) be an arbitrary sequence of positive numbers such that \(\lim_{k \to +\infty} R_k=0\). Furthermore, let \(E_k:= R_k^{-1}(E-x)\), and, by rotating, assume without loss of generality that \(\nu_E(x)=-e_n\). By \cite[Theorem 15.5]{maggi_sets_2012}, \(E_k \to \R^{n}_+\) in \(L^1_{\mathrm{loc}}(\R^n)\) and \(\mathcal H^{n-1}\mres \partial^\ast E_k \stackrel{\ast}{\rightharpoonup}\mathcal H^{n-1}\mres  \partial \R^n_+\), so, in particular, \begin{align*}
       \operatorname{Per}(\R^n_+\setminus E_k ; B)+  \operatorname{Per}(E_k\setminus \R^n_+ ; B) \leq   2\operatorname{Per}(E_k ; B) +  2\operatorname{Per}(\R^n_+ ; B) \leq C \qquad \text{for all balls } B\subset \R^n
    \end{align*} where \(C>0\) is independent of \(k\). Hence, \(E_k\) and \(E=\R^n_+\) satisfy the assumptions of \thref{aT0zFAPm}, so we have \begin{align*}
         \lim_{k\to+\infty }\Phi_{E,x}(R_k)= \lim_{k\to+\infty }\Phi_{E_k,0}(1) = \Phi_{\R^n_+,0}(1)
    \end{align*} which implies~\eqref{bpwPOxAA}. Moreover, from~\thref{PFeFnfBd} and the properties of the Gamma function, \begin{align*}
        \Phi_{\R^n_+,0}(1) &=\bigg ( \frac{2 \pi^{\frac n 2 - 1 } \Gamma \big ( \frac{s+1}2 \big )\Gamma \big ( \frac{1-s}2 +1\big ) }{\Gamma \big ( \frac s2+1 \big )\Gamma \big ( \frac{n-s}2+1 \big )} \bigg ) \frac s {1-s} \geq  \frac {C(n)s} {1-s} 
    \end{align*} using that the Gamma function is continuous when restricted to the positive reals. 
\end{proof}

\section{Proof of main theorem}

In this section, we give the proof of \thref{jG1dQXin}.

\begin{proof}[Proof of \thref{jG1dQXin}] 
   By translating, we may assume, without loss of generality, that \(x=0\in \partial^\ast E\). Let \(s>s_0>0\) and let \(U_E\) be given by~\eqref{kzYTNNBU}. From \thref{uUBAsC54} and \thref{GeFXLo81}, we have that \begin{align*}
        \int_{\tilde B_R^+ } x_{n+1}^{1-s} \vert \nabla U_E \vert^2 \dd X \geq  R^{n-s}  \lim_{r\to 0^+} \Phi_{E,0}(r) \geq \frac{CR^{n-s}}{1-s}
    \end{align*} for all \(R>0\) with \(C>0\) depending only on \(n\) and \(s_0\). Consequently, by~\thref{Qu1H6WRf}, we have \begin{align*}
        \frac{CR^{n-s}}{1-s} \leq \operatorname{Per}_s(E;B_{2R}) \qquad \text{for all }R>0.
    \end{align*} Next, from \thref{60sDucQ0} and the relative isoperimetric inequality 
    \begin{align*}
        \operatorname{Per}_s(E; B_R) \leq  \frac {C(n)} s  \bigg (\frac{1}{1-s} \big(\operatorname{Per}(E;B_{2\varepsilon^{- 1/s  }R}) \big)^{\frac{n-s}{n-1}}  + \varepsilon  R^{n-s} \bigg ) 
    \end{align*} for all \(\varepsilon\in (0,3^{-s})\), so  \begin{align*}
         \frac{R^{n-s}}{1-s} \leq C  \bigg (\frac{1}{1-s} \big(\operatorname{Per}(E;B_{4\varepsilon^{- 1/s  }R}) \big)^{\frac{n-s}{n-1}}  + \varepsilon  R^{n-s} \bigg ) .
    \end{align*} Choosing \(\varepsilon = \min \{3^{-s},2^{-1}C^{-1}(1-s)^{-1} \}\), we obtain \begin{align*}
        R^{n-1} \leq C \operatorname{Per}(E;B_{4\varepsilon^{- 1/s  }R}) \qquad \text{for all }R>0
    \end{align*} Replacing \(R\) with \(4^{-1}\varepsilon^{1/s}R\) and using that \(\varepsilon^{1/s}\leq C(n,s_0)\), we conclude \begin{align*}
        R^{n-1} \leq C \operatorname{Per}(E;B_R) \qquad \text{for all }R>0
    \end{align*} with \(C\) robust as \(s\to 1^-\).
    
\end{proof}

\section{Declarations}

The authors have no competing interests to declare that are relevant to the content of this article. This manuscript has no associated data and conforms to relevant ethical standards. The author is supported by an Australian Government Research Training Program Scholarship.

\vfill
\printbibliography

\end{document}